\documentclass[12pt]{amsart}

\usepackage[dvipsnames]{xcolor}
\colorlet{LightRubineRed}{RubineRed!70}
\colorlet{Mycolor1}{green!10!orange}
\definecolor{Mycolor2}{HTML}{00F9DE}

\usepackage{amsfonts}
\usepackage{amsmath}
\usepackage{amssymb}
\usepackage{amsthm} 
\usepackage{mathrsfs}
\usepackage{xcolor}
\usepackage{appendix}
 \usepackage{hyperref}

\usepackage{graphicx}
\def\pprec{\mathrel{\scalebox{.9}[1]{$\prec$}\mkern-5mu%
  \scalebox{.4}[1]{$\prec$}\mkern-5.5mu\scalebox{.4}[1]{$\prec$}}}

\newcommand\utimes{\mathbin{\ooalign{$\cup$\cr%
   \hfil\raise0.42ex\hbox{$\scriptscriptstyle\times$}\hfil\cr}}}
\newcommand\bigutimes{\mathop{\ooalign{$\bigcup$\cr%
   \hfil\raise0.36ex\hbox{$\scriptscriptstyle\boldsymbol{\times}$}\hfil\cr}}}
\usepackage[T1]{fontenc}
\usepackage{mathabx,graphicx}

\theoremstyle{definition}
\newtheorem{thm}{Theorem}[section] 
\newtheorem{cor}[thm]{Corollary}

\newtheorem{prop}[thm]{Proposition}
\theoremstyle{definition}
\newtheorem{defn}[thm]{Definition}
\newtheorem{rem}[thm]{Remark}

\newcommand{\diag}{\mathrm{Diag}}
\newcommand{\E}{\mathbb{E}}

\newcommand{\cA}{\mathcal{A}}

\newcommand{\bC}{\mathbb{C}}
\newcommand{\cB}{\mathcal{B}}
\newcommand{\cC}{\mathcal{C}}

\let\phi=\varphi
\let\epsilon=\varepsilon

\numberwithin{equation}{section}

\title{The Distribution of polynomials in monotone independent elements}

\author[M. Banna]{Marwa Banna}
\address{New York University Abu Dhabi, Division of Science, Mathematics, Abu Dhabi, UAE}
\email{marwa.banna@nyu.edu}

\author[PL. Tseng]{Pei-Lun Tseng}
\address{New York University Abu Dhabi, Division of Science, Mathematics, Abu Dhabi, UAE}
\email{pt2270@nyu.edu}

\date{\today}


\begin{document}

\begin{abstract}

Building on the work of  Arizmendi and Celestino (2021), we derive the $*$-distributions of polynomials in monotone independent and infinitesimally monotone independent elements. For non-zero complex numbers $\alpha$ and $\beta$, we derive explicitly the $*$-distribution of $p_{\alpha,\beta}=\alpha ab + \beta ba$ whenever $a$ and $b$ are monotone or infinitesimally monotone independent elements. This encompasses both cases of the commutator and anti-commutator. This approach can be pushed to study more general polynomials. As applications, we derive the limiting distribution with respect to the partial trace of polynomials in a certain class of random matrices.

\end{abstract}

\maketitle

\date{}

\section{Introduction}

In non-commutative probability theory, diverse notions of independence arise through the examination of purely algebraic representations of fundamental concepts in classical probability theory. Five notions of independence were identified to satisfy natural axioms of probability: tensor (classical), free, Boolean, monotone, and its mirror image, anti-monotone. These notions of independence were classified by Speicher \cite{s97} and Muraki \cite{mu02,M03} and also by Ben Ghorbal and Schürmann \cite{BG-Sch} in the framework of category theory.

Following Voiculescu's introduction of the concept of free independence, numerous extensions and generalizations of freeness have arisen within the non-commutative setting. Among these, one notable generalization is infinitesimal freeness, see \cite{FN10} and \cite{BGN03}. This extension has found practical applications in investigating the occurrence of spikes in diverse deformed random matrix models, as shown in \cite{Au21, SH18, CDG22}. Consequently, extensions of infinitesimal Boolean and monotone independence have also been introduced in  \cite{PT21, H11}.

Just like the classical counterpart, these notions of independence enable the computation of probability distributions for sums and products of independent random variables when we have knowledge of their individual distributions. This is achieved through the corresponding operations of additive and multiplicative convolutions. For more details, we refer to \cite{V85, V87} for free convolutions, \cite{SW97, F04} for Boolean convolutions, and \cite{M2000, B05} for monotone convolutions. Similarly, analogous convolutions, known as infinitesimal additive and multiplicative convolutions, have been developed in  \cite{BS12, M19, P10, TS19} for the free case and \cite{PT21, TS22} for the Boolean and monotone cases.

Beyond sums and products, there is a keen interest in determining distributions of general polynomials in independent variables. A natural starting point involves deriving the distributions of two particular polynomials: the commutator $i(x_1x_2-x_2x_1)$, and the anti-commutator, $x_1x_2+x_2x_1$, for independent variables $x_1$ and $x_2$.  While the distribution of the commutator in the free setting was derived combinatorially using cumulants in prior works \cite{NS06}, the distribution of the anti-commutator was only obtained recently in \cite{FMNS20,P21}. This is also the case for Boolean independent variables where the distributions of both the commutator and anti-commutator were determined recently in \cite{MT23}. Furthermore, the infinitesimal free and Boolean distributions of commutators and anti-commutators were, under certain assumptions, derived recently in \cite{MT23}.

On the other hand, distributions of commutators and anti-commutators for monotone independent variables and infinitesimally monotone independent remain unexplored. One of our primary objectives in this paper is to address this gap. As polynomials, we study closely linear spans in Section \ref{Section monotone polynomials} for which we obtain explicit results. In fact, let $(\mathcal{C},\phi)$ be a $*$-probability space and $a$, $b \in \cC$  such that $\cA_a\prec \mathring{\cA}_b$  where $\cA_a$ is the $*$-algebra generated by $a$ and $\mathring{\cA}_b$ is the $*$-algebra generated by $b$ and $1$. We determine explicitly in Theorem \ref{thm: M-linear span} the $*$-distribution of $p_{\alpha,\beta}=\alpha ab + \beta ba$, where $\alpha$ and $\beta$ are non-zero complex numbers. Theorem \ref{thm: M-linear span} yields for the particular case where  $\phi(b)=0$ that for any $k \geq 1$,
$$
\phi(p_{\alpha,\beta}^k)=\begin{cases}
\sqrt{\alpha\beta \phi(b^2)}^k\phi(a^k) & \text{ if }  k \text{ is even,} \\ 0  & \text{ if }  k \text{ is odd.} \end{cases}$$
The framework of Theorem \ref{thm: M-linear span} encompasses both cases of the commutator $q=i(ab-ba)$ and anti-commutator $p=ab+ba$ and yields for each $k\geq 1,$
\begin{eqnarray*}
\varphi(p^k)&=& \frac{1}{2\sqrt{\varphi(b^2)}}\left([\varphi(b)+\sqrt{\varphi(b^2)}]^{k+1}-[\varphi(b)-\sqrt{\varphi(b^2)}]^{k+1}\right)\varphi(a^k); \\
\varphi(q^k)&=& \frac{1}{2}\left(1-(-1)^{k+1}\right)\left(\sqrt{\varphi(b^2)-\varphi(b)^2}\right)^{k}\varphi(a^k).
\end{eqnarray*}
Moreover, using the upper triangular technique developed in \cite{TS19,PT21}, we extend our results to the infinitesimal monotone setting where we study linear spans in Section \ref{Section inf linear span} and general polynomials as described in Remark \ref{remark inf general poly}.  

In the work of Lenczewski \cite{l14}, random matrix models for asymptotic monotone independence with respect to the partial trace were given and were further discussed in \cite{CLS22,CN22}. On the other hand, Mingo and Tseng \cite{MT23} introduced recently a constructive approach for generating random matrix models for which asymptotic monotone independence holds. We extend in this paper the utility of our results to derive the distribution of such random matrices with respect to the partial trace, as elaborated in Section \ref{Section: RM example}.

Our approach builds upon the work of Arizmendi and Celestino \cite{AC21} who study distributions of polynomials with cyclic monotone elements. The notion of cyclic monotone independence  was introduced by Collins, Hasebe, and Sakuma \cite{CHS18} to describe the joint limiting spectral distribution of some pairs of matrices. They derived spectral formulas for certain polynomials of degree $2$ and $3$ in cyclic monotone elements, including the commutator and anti-commutator. In \cite{AC21}, Arizmendi and Celestino offered a novel approach for finding spectral distributions for polynomials in cyclic monotone elements. A pivotal aspect of their method lies in the observation that many polynomials can be expressed as the $(1,1)$-entry of matrix products, allowing for the derivation of the distribution by examining the corresponding product matrices. Building on this idea, we extend this approach to study polynomials in monotone independent elements. Moreover, this approach can be further pushed to investigate more general polynomials which lead to computing the precise distributions of a wide range of polynomials in random matrices as Proposition \ref{prop Wigner general poly}.

\tableofcontents

\section{Preliminaries}

\subsection{Monotone Independence}

A \emph{non-commutative probability space} (or ncps for short) is a pair $(\cC,\varphi)$ such that $\cC$ is a unital algebra and $\varphi$ is a linear functional on $\cC$ that $\varphi(1)=1$. Moreover, a ncps $(\cC,\phi)$ is called a \emph{$*$-probability space} if we further assume $\cC$ is a $*$-algebra and $\phi(a^*a)\geq 0$ for all $a\in \cC$. 

For $a_1,\dots,a_k\in \cC$, the \emph{$*$-distribution} of $\{a_1,\dots,a_k\}$ is the set of all possible joints moments  
$$
\mu_{a_1,\dots,a_k}=\{\varphi(a_{i_1}^{m_1}\cdots a_{i_n}^{m_n})\mid n\geq1,\ 1\leq i_1,\dots,i_n\leq k\text{ and }m_1,\dots,m_n\in \{1,*\}  \}.
$$
\begin{defn}
Suppose that $(\mathcal{C},\varphi)$ is a $*$-probability space, and $\mathcal{A}$ and $\cB$ are $*$-subalgebras of $\mathcal{C}$ that are not necessarily unital. We say that  $(\mathcal{A},\mathcal{B})$ is \emph{monotone independent} and write  $\mathcal{A} \prec \mathcal{B}$ if 
\begin{eqnarray}
\phi(b_0a_1b_1\cdots a_nb_n) &=& \phi(a_1\cdots a_n)\phi(b_0)\phi(b_1)\cdots \phi(b_n),
\end{eqnarray}
for all $a_1,\dots,a_n\in \cA$, $b_1,\dots,b_{n-1}\in \cB$ and $b_0,b_n$ are either $1$ or elements in $\cB$.
\end{defn}

For two sets $A_1$ and $A_2$ of elements in $\mathcal{C}$, we say $(A_1,A_2)$ is monotone independent and write $A_1\prec A_2$ if $\cA_{1}\prec \mathring{\cA}_2$ where $\cA_1$ is the $*$-algebra generated by $A_1$ and $\mathring{\cA}_2$ is the $*$-algebra generated by $A_2$ and $1$.  

\begin{rem}
Note that under the notion of monotone independence of pair algebras $(\cA_1,\cA_2)$,  we usually don't consider $\cA_1$ to be unital; otherwise, for all $a\in \cA_2,$ we have
$$
\phi(a^n) = \phi(1\cdot a \cdot 1 \cdot a \cdots 1 \cdot a) = \phi(a)^n\phi(1\cdot 1 \cdots 1\cdot 1)=\phi(a)^n,
$$
which is not the case that we are interested in. 
\end{rem}

\subsection{Infinitesimally Monotone Independence}\label{subsection: inf Monotone}

Let $(\cC, \phi)$ be a $*$-probability space and $\phi': \cC \rightarrow \bC$ be a linear functional such that $\phi'(1) = 0$ and $\phi(a^*)=\overline{\phi(a)}$ for all $a\in\cC$. Then we call the triple $(\cC, \phi, \phi')$ an \textit{infinitesimal $*$-probability space}.

For $a_1,\dots,a_k\in \cC$, the \emph{infinitesimal $*$-distribution} of $\{a_1,\dots,a_k\}$ is the
pair $(\mu_{a_1,\dots,a_k},\mu'_{a_1,\dots,a_k})$ where 
$\mu_{a_1,\dots,a_k}$ is the $*$-distribution of $\{a_1,\dots,a_k\}$ and 
$$
\mu'_{a_1,\dots,a_k}=\{\varphi'(a_{i_1}^{m_1}\cdots a_{i_n}^{m_n})\mid n\geq1,\ 1\leq i_1,\dots,i_n\leq k\text{ and }m_1,\dots,m_n\in \{1,*\}  \}.
$$
\begin{defn}
Let $(\cC,\varphi,\varphi')$ be an infinitesimal $*$-probability space and let $\cA$ and $\cB$ be $*$-subalgebras of $\cC$ that are not necessary unital. We say that $(\mathcal{A},\mathcal{B})$ is \emph{infinitesimally monotone independent} and write $\mathcal{A} \pprec \mathcal{B}$ if
\begin{eqnarray*}
\phi(b_0a_1b_1\cdots a_nb_n) &=& \phi(a_1\cdots a_n)\phi(b_0)\phi(b_1)\cdots \phi(b_n), \\
\phi'(b_0a_1b_1\cdots a_nb_n) &=& \phi'(a_1\cdots a_n)\phi(b_0)\phi(b_1)\cdots \phi(b_n) \\
&&+ \phi(a_1\cdots a_n)\sum_{j=0}^n \phi(b_0)\cdots \phi(b_{j-1})\phi'(b_j)\phi(b_{j+1})\cdots \phi(b_n),
\end{eqnarray*}
for all $a_1,\dots,a_n\in \cA$, $b_1,\dots,b_{n-1}\in \cB$ and $b_0,b_n$ are either $1$ or elements in $\cB$.
\end{defn}
For two sets $A_1$ and $A_2$ in $\mathcal{C}$, we say that $(A_1,A_2)$ is infinitesimally monotone independent and write $A_1\pprec A_2$ if $\cA_1\pprec \mathring{\cA}_2$ with $\cA_1$ being the $*$-algebra generated by $A_1$ and $\mathring{\cA}_2$ be the $*$-algebra generated by $A_2$ and $1$.

Note that the infinitesimal monotone independence can be understood as monotone independence with amalgamation. To be precise, let $(\cC,\phi,\phi')$ be an infinitesimal $*$-probability space and define 
$$
\widetilde{\cC}=\Big\{\begin{bmatrix}
 a & a' \\ 0 & a   
\end{bmatrix}  \Big | a,a'\in\mathcal{C}  \Big\},\  \widetilde{\mathbb{C}}=\Big\{\begin{bmatrix}
 c & c' \\ 0 & c   
\end{bmatrix}  \Big | c,c'\in\mathbb{C}  \Big\}.
$$
Finally, let  $\widetilde{\phi}:\widetilde{\cC}\to\widetilde{\mathbb{C}}$ be the linear functional defined by
$$\widetilde{\phi}\Big(\begin{bmatrix}
 a & a' \\ 0 & a   
\end{bmatrix}\Big)=\begin{bmatrix}
    \phi(a) & \phi'(a)+\phi(a') \\ 0 & \phi(a)
\end{bmatrix}.
$$
Then the triple $(\widetilde{\cC},\widetilde{\mathbb{C}},\widetilde{\phi})$ forms a $\widetilde{\mathbb{C}}$-valued probability space (see \cite{TS19}). The correspondence to independence with amalgamation is illustrated by the following theorem by Perales and Tseng \cite{PT21}.
\begin{thm}\label{thm: inf-M vs OV-M}
Let $(\cC,\varphi,\varphi')$ be an infinitesimal $*$-probability space and let $\cA$ and $\cB$ be $*$-subalgebras of $\cC$ that are not necessarily unital. Then  $(\cA,\cB)$ is infinitesimally monotone independent if and only if $(\widetilde{\cA},\widetilde{\cB})$ is monotone independent with respect to $\widetilde{\varphi}$ where 
$$
\widetilde{\cA}=
\Big\{\begin{bmatrix}
 a & a' \\ 0 & a   
\end{bmatrix}  \Big | a,a'\in\mathcal{A}  \Big\}
\text{ and } \widetilde{\cB} = \Big\{\begin{bmatrix}
 a & a' \\ 0 & a   
\end{bmatrix}  \Big | a,a'\in\mathcal{B}  \Big\}.$$
\end{thm}

\subsection{Random Matrices for Monotone independence}\label{subsection 2.3}

\begin{defn}
Let $(\cC_n,\phi_n)_n$ be a sequence of ncps and $(a_n,b_n)$ be a pair of elements in $\cC_n$ for each $n$. We say that the sequence $(a_n,b_n)_n$ is \emph{asymptotically monotone independent} with respect to $\phi_n$ if there exist a ncps $(\cC,\phi)$ and elements $a,b\in \cC$ with $a\prec b$ such that 
$(a_n,b_n)$ converges to $(a,b)$ in distribution; that is,
$$
\lim_{n\to\infty} \phi_n(p(a_n,b_n))=\phi(p(a,b)) \text{ for any polynomial } p
$$
and $a \prec b $  with respect to $\phi$. 

\end{defn}
We present the random matrix model given by Lenczewski \cite{l14} where asymptotic monotone independence holds with respect to the partial trace. Let $A_N$ be an $N\times N$ random matrix and $N_0\leq N$ be a fixed natural number. $A_N$ can be represented as a composition of four block matrices in the following manner:
$$
A_N=\begin{bmatrix}
    A^{(1,1)} & A^{(1,2)} \\ A^{(2,1)} & A^{(2,2)}
\end{bmatrix},
$$
where $A^{(1,1)}$ corresponds to a square submatrix of size $N_0\times N_0$, $A^{(2,2)}$  to a square submatrix of size $(N-N_0)\times (N-N_0)$, $A^{(1,2)}$  to a submatrix of size $N_0\times (N-N_0)$, and $A^{(2,1)}$  to a submatrix of size $(N-N_0)\times N_0$. The partial trace $\psi_N$ of $A_N$ is then defined by
\begin{equation}\label{formula: partial-trace}
\psi_N(A_N):= \frac{1}{N_0}(\E\circ Tr_{N_0})(A_N^{(1,1)})
\end{equation}
where $Tr_{N_0}$ is the non-normalized trace. 
Here we note that the partial trace $\psi_N$ depends on $N_0$. Lenczewski showed that if $A_N$ and $B_N$ are independent GUE matrices, then ($T_{A_N},B_N)$ is asymptotically monotone independent with respect to $\psi_N$ where 
$$T_{A_N}=\begin{bmatrix}
    0 & A^{(1,2)} \\ A^{(2,1)} & 0 
\end{bmatrix}.$$

In \cite{MT23}, Mingo and Tseng considered the notion of infinitesimal freeness and  showed a method for utilizing infinitesimal idempotents to construct monotone independent random matrix models. By this construction, they recover the random matrix models in \cite{l14} as particular cases. We illustrate this construction for the reader's convenience and start by recalling some notation.

Let $(\cC,\phi,\phi')$ be an infinitesimal probability space, and let $j\in\cC$ be an \emph{infinitesimal idempotent} element, that is,  $\phi(j^k)=0$ for all $k\geq 1$ and $j=j^2$. Assume furthermore that $\phi'(j)\neq 0$ and let $\cA$ be a unital subalgebra of $\cC$ that is infinitesimally free from $\{j\}$. Finally, define $$\psi(a)=\frac{1}{\phi'(j)}\phi'(aj) \text{ for all }a\in\cC,$$
and note that $\psi(a)=\varphi(a)$ for all $a\in\cA.$
It is clear that $\psi$ is a linear functional on $\cC$ with $\psi(1)=1$, and hence, $(\cC,\psi)$ is a ncps. Finally, set  $j^{(-1)}:=j$ and $j^{(1)}:=1-j$, and define $\mathcal{J}_a(\cA)$ to be the algebra generated by
$$
\{j^{(\epsilon_1)}a_1j^{(\epsilon_2)}a_2\cdots j^{(\epsilon_{k-1})}a_{k-1}j^{(\epsilon_k)}\mid k\geq 0, \epsilon_1,\dots,\epsilon_k\in\{\pm 1\} \text{ with }\epsilon_1\neq \cdots \neq \epsilon_k, a_1,\dots,a_k\in \cA\}.
$$
We refer to  \cite{FN10, M19} for more details on infinitesimal free independence and recall now the following result in \cite[Theorem 5.5]{MT23}.
\begin{thm}\label{thm: Inf idem to M}
Suppose $\cA$ and $\cB$ are unital subalgebras of $\cC$ such that $\{\cA,\cB\}$ are infinitesimally free from $j$  in $(\cC,\varphi,\varphi')$. If $\cA$ and $\cB$ are free with respect to $\varphi$, then 
$(\mathcal{J}_a(\cA),\cB)$ is monotone independent  with respect to $\psi$. 
\end{thm}

 Let us recall a natural framework of an infinitesimal $*$-probability space for random matrices. Let $A_N^{(1)},\dots,A_N^{(k)}$ be $N\times N$ random matrices, and consider the linear map $\varphi_N$ on $\mathbb{C}\langle X_1,\dots,X_k\rangle$, the algebra of polynomials in $k$ non-commuting indeterminates,  defined by
$$
\varphi_N(p)= \frac{1}{N} (\E\circ Tr_N)\big(p(A_N^{(1)},\dots,A_N^{(k)})\big). 
$$
If  for all $p\in \mathbb{C}\langle X_1,\dots,X_k\rangle$,  $\varphi(p):=\lim_{N\to\infty}\varphi_N(p)$ exists as well as
$\phi'(p)=\lim_{N\to\infty}\phi_N'(p)$ where
$\varphi_N'(p):= N[\varphi_N(p)-\varphi(p)]$, then $(\varphi,\varphi')$ is said to be the asymptotic infinitesimal law of $\{A_N^{(1)},\dots,A_N^{(k)}\}$. 

 We recall now real and complex Wigner matrices and consider  certain assumptions on their moments. We will  provide applications of our main results to polynomials in such matrices in Section \ref{Section: RM example}. We denote by $W_N$ the $N \times N$ Hermitian matrix defined by $W_N=\frac{1}{\sqrt{N}}(W_{i,j}^{(N)})_{1\leq i, j \leq N}$ with $W_{i,j}^{(N)}=\widebar{W}_{j,i}^{(N)}$ for $j>i$. For each $N \in \mathbb{N}$, $(W_{i,j}^{(N)})_{1\leq i\leq j \leq N}$  is a family of independent random variables such that the diagonal entries are real-valued centered random variables with common variance, say $s^2$. For real Wigner matrices, we assume that off-diagonal entries are real-valued variables satisfying for all $i<j $:
\[
\E[W_{i,j}^{(N)}]=0 \quad \text{and} \quad  \E[|W_{i,j}^{(N)}|^2]= \sigma^2.
\]
For the complex Wigner matrices, we assume that the off-diagonal entries are complex-valued and satisfy, in addition to the above conditions,  $\E[(W_{i,j}^{(N)})^2]=0$ for all $i<j$.
 It is well-known that the limiting spectral distribution of a  Wigner matrix is the semicircular distribution. Since the latter is compactly supported, this is equivalent to the convergence of the moments, i.e., for all $k\geq 1$,  
$$
\lim_{N\to\infty}\frac{1}{N}(\E\circ Tr_N)(W_N^{k}) = \begin{cases}
    C_{m} & \text{ if }k=2m \text{ is even}, \\
    0 & \text{ if }k \text{ is odd,}
\end{cases}
$$
where $C_m$ is the $m$-th Catalan number. Moreover, a family of independent real or complex Wigner matrices is known to be asymptotically free with respect to $\varphi_N$, see for instance \cite{AGZ}[Theorem 5.4.2]. Now  assuming furthermore that for $i<j$, $\E[|W_{i,j}^{(N)}|^4]=\alpha$ and 
\[
\sup_{N\in \mathbb{N}} \sup_{1 \leq i\leq k \leq N} \E[|W_{i,j}^{(N)}|^\ell]<\infty \quad \text{for all }\quad \ell \in \mathbb{N},
\]
then Au \cite{Au21} proved that that $W_N$ has an infinitesimal distribution generalizing the result of \cite{NL16} for matrices whose entries are identically distributed. Moreover, Au proves that a family of independent Wigner matrices satisfying the above moment conditions is asymptotically infinitesimally free from finite rank matrices. Indeed, Lemma 3.4 and Corollary 3.11 in \cite{Au21} yield the following:
\begin{thm}\label{thm:inf free - WEJ} 
Let $(W_N^{(i)})_{i\in I}$ be a family of independent $N\times N$ real or complex Wigner matrices satisfying the above moment conditions and denote by $E_{N}^{(j,k)}$ the unit matrix in the $(j,k)$-th coordinate. Then for any fixed $N_0$, $(W_N^{(i)})_{i\in I}$ and $(E_N^{(j,k)})_{1 \leq j,k \leq N_0}$ are asymptotically infinitesimally free  with respect to $(\phi_N,\phi_N')$. 
\end{thm}
We note that in a more recent paper, Au proves asymptotic infinitesimal freeness for a more general class of matrices, see \cite[Proposition 3.6]{Au23}.

It is noteworthy that the limit operators of  $(E_N^{(j,j)})_{1\leq j\leq N_0}$ are infinitesimal idempotents. 
Consequently, we can construct many asymptotically monotone independent random matrix models by combining Theorem \ref{thm: Inf idem to M} and Theorem \ref{thm:inf free - WEJ}. In particular, choose $A_N$ and $B_N$ to be asymptotically free and asymptotically infinitesimally free from
$j_N=\sum_{j=1}^{N_0}E_{N}^{j,j}$. It is easy to see that 
$\psi_N(\cdot)=\frac{1}{\varphi_N'(j_N)}\varphi_N'(\cdot j_N)$ is the partial trace in  \eqref{formula: partial-trace} and that 
$T_{A_N}=j_NA_N(1-j_N)+(1-j_N)A_N j_N.$  This yields that $(T_{A_N}, B_N)$ is asymptotically monotone independent with respect to the partial trace. Moreover, the limiting distribution of $T_{A_N}$ with respect to $\psi_N$ can then be computed following \cite[Proposition 5.3]{MT23} which, whenever $A_N$ and $j_N$ are asymptotically infinitesimally free, yields that 
\begin{eqnarray}\label{moments T_A_N}
\lim_{N\to \infty}\psi_N(T_{A_N}^k) 
&=& 
\begin{cases}\Big[\lim\limits_{N\to \infty}\big( (\E\circ Tr_N)(A_N^2)-(\E\circ Tr_N)(A_N)^2\big)\Big]^{k/2} & \text{ if }  k \text{ is even,} \\
            0  & \text{ if }  k \text{ is odd,} 
        \end{cases}
\end{eqnarray}
where $tr_N$ denotes the normalized trace.

\section{Distribution of Polynomials in Monotone Independent Elements}\label{Section monotone polynomials}

In this section, we extend the approach  in \cite{AC21} to the monotone setting and study distributions of polynomials in monotone independent elements. For two monotone independent elements $a$ and $b$, we start by studying distributions of linear spans of $ab$ and $ba$ for which we derive explicit expressions. As particular cases, we determine the distributions of the commutator and anti-commutator. Then, in Remark \ref{remark general poly}, we discuss more general polynomials in monotone independent elements.

\medskip
\noindent
\paragraph{\bf The Linear Span of $ab$ and $ba$}
Let $(\mathcal{C},\phi)$ be a $*$-probability space. For any $a$ and $b$ in $\cC$, we denote by $p_{\alpha,\beta}:=p_{\alpha,\beta}(a,b)$, the linear span of $ab$ and $ba$, given by $p_{\alpha,\beta}=\alpha ab+\beta ba$ for some non-zero complex numbers $\alpha$ and $\beta$. 

 We recall that for a given $x\in\cA$, we set $\cA_x$ to be the $*$-algebra generated by $x$, and $\mathring{\cA}_x$ to be the $*$-algebra generated by $x$ and $1$.  
\begin{thm}\label{thm: M-linear span}
 Assume that $\cA_a\prec \mathring{\mathcal{A}}_b$, then for each $k\geq 1$, we have for any $k \in \mathbb{N}$,
    \begin{eqnarray*}
\varphi(p_{\alpha,\beta}^k)=\frac{\varphi(a^k)}{2^{k+1}\gamma} \left[\left((\alpha+\beta)x+\gamma\right)^{k+1}-\left((\alpha+\beta)x-\gamma\right)^{k+1} \right]
    \end{eqnarray*}
where $x=\varphi(b)$ and $\gamma=\sqrt{(\alpha-\beta)^2\varphi(b)^2+4\alpha\beta\varphi(b^2)}$.
\end{thm}
\begin{rem}\label{rem: dist of span}
Note that by monotone independence, the $*$-distribution of $p_{\alpha,\beta}$ depends on the distribution of $b$ only through its first and second moments $\phi(b)$ and $\phi(b^2)$. Now, if we assume that $\phi(b)=0$, then for any $k \geq 1$,
$$
\phi(p_{\alpha,\beta}^k)=\frac{1}{2}\sqrt{\alpha\beta \phi(b^2)}^k\phi(a^k)+\frac{(-1)^k}{2}\sqrt{\alpha\beta \phi(b^2)}^k\phi(a^k)=\begin{cases}
            \sqrt{\alpha\beta \phi(b^2)}^k\phi(a^k) & \text{ if }  k \text{ is even,} \\
            0  & \text{ if }  k \text{ is odd.} 
        \end{cases}$$
Furthermore, if $a$ is an even operator, then we have the $*$-distribution of $p_{\alpha,\beta}$ is a mere scaling of the distribution of $a$. 

\end{rem}

Note that, as a direct consequence, we obtain the $*$-distribution of anti-commutators and commutators in monotone independent variables by taking $(\alpha,\beta)=(1,1)$ or $(\alpha,\beta)=(i,-i)$ respectively. More precisely, we obtain the following moment expressions.
\begin{cor}\label{cor: anti & comm - M}
Consider $a$, $b \in \cC$  such that $\cA_a\prec \mathring{\cA}_b$. Let $p=ab+ba$ and $q=i(ab-ba)$, then for each $k\geq 1,$
\begin{eqnarray*}
\varphi(p^k)&=& \frac{1}{2\sqrt{\varphi(b^2)}}\left([\varphi(b)+\sqrt{\varphi(b^2)}]^{k+1}-[\varphi(b)-\sqrt{\varphi(b^2)}]^{k+1}\right)\varphi(a^k); \\
\varphi(q^k)&=& \frac{1}{2}\left(1-(-1)^{k+1}\right)\left(\sqrt{\varphi(b^2)-\varphi(b)^2}\right)^{k}\varphi(a^k).
\end{eqnarray*}
\end{cor}

In order to derive explicit expressions for the distributions of spans of $ab$ and $ba$, we prove first the following proposition that plays a crucial role in bridging the gap between matrices with monotone entries and the distribution of polynomials. For a given algebra $\mathcal{D}$, we represent the unital $*$-algebra generated by $\mathcal{D}$ as $\mathring{\mathcal{D}}$.

\begin{prop}\label{prop: Tr-monotone}
Let $(\mathcal{C},\varphi)$ be a $*$-probability space. Suppose that $\mathcal{A}$ and $\mathcal{B}$ are $*$-subalgebras of $\mathcal{C}$ such that $\mathcal{A}\prec \mathring{\mathcal{B}}$. Consider $A_k=[a_{i,j}^{(k)}]_{i,j\in [n]}\in M_n(\mathcal{A})$ for each $k=1,\dots,m$ and $B_s=[b_{i,j}^{(s)}]_{i,j\in [n]}\in M_n(\mathcal{B})$ for each $s=0,1,\dots,m$. Then we have 
\begin{equation*}
(Tr_n\otimes \varphi)(D_0A_1B_1A_2B_2\cdots B_{m-1}A_mD_m)=(Tr_n\otimes \varphi)(D_0'A_1B_1'A_2B_2'\cdots B_{m-1}'A_mD_m')
\end{equation*}
where $D_0=I_n $ or $B_0$, $D_m=I_n$ or $B_m$, 
$
B_k'=\left(\varphi(b_{i,j}^{(k)})_{i,j\in [n]}\right)\in M_n(\mathbb{C}),
$ for each $k$, and $D_0'=I_n$ or $B_0'$, $D_m'=I_n$ or $B_m'$.
\end{prop}

\begin{proof}
Without loss of generality, we may assume $D_0=I_n$ and $D_m=B_m$ (the proof of all other cases are similar). Note that by monotone independence with respect to $\varphi$, we obtain
\begin{align*}
(Tr_n\otimes \varphi)& (A_1B_1A_2B_2\cdots A_mB_m)\\&= \sum_{i_1,\dots,i_{2m}\in [n]}\varphi\left(a_{i_1,i_2}^{(1)}b_{i_2,i_3}^{(1)}a_{i_3,i_4}^{(2)}b_{i_4,i_5}^{(2)}\cdots a_{i_{2m-1},i_{2m}}^{(m)}b_{i_{2m},i_1}^{(m)}\right) \\
&=  \sum_{i_1,\dots,i_{2m}\in [n]}\varphi\left(a_{i_1,i_2}^{(1)}\varphi(b_{i_2,i_3}^{(1)})a_{i_3,i_4}^{(2)}\varphi(b_{i_4,i_5}^{(2)})\cdots a_{i_{2m-1},i_{2m}}^{(m)}\varphi(b_{i_{2m},i_1}^{(m)})\right) \\
&= (Tr_n\otimes \varphi) (A_1B_1'A_2B_2'\cdots A_mB_m'). 
\end{align*}

\end{proof}

Having this in hand, we are ready to prove the above results.
\begin{proof}[Proof of Theorem \ref{thm: M-linear span}]
Let
$$
B_0=\begin{bmatrix}
\alpha & b \\ 0 & 0
\end{bmatrix},\ A=\begin{bmatrix}a & 0 \\ 0 & a\end{bmatrix}, \text{ and }B_1=\begin{bmatrix} b & 0 \\ \beta & 0
\end{bmatrix}.
$$
Then it is obvious that $B_0AB_1=\diag(\alpha ab+\beta ba,0)$. Note that following Proposition \ref{prop: Tr-monotone}, for any $k\geq 1$, we have
\begin{eqnarray*}
\varphi(p_{\alpha,\beta}^k)&=& (Tr_2\otimes \varphi)\left((B_0AB_1)^k\right) \\
&=& (Tr_2\otimes \varphi) \left( B_0(AB_1B_0)^{k-1}AB_1\right) \\
&=& (Tr_2\otimes \varphi) \left(\begin{bmatrix}
\alpha & \varphi(b) \\ 0 & 0
\end{bmatrix}\left(\begin{bmatrix}
    a & 0 \\ 0 & a
\end{bmatrix}\begin{bmatrix}\alpha\varphi(b) & \varphi(b^2)\\ \alpha\beta & \beta\varphi(b)\end{bmatrix}\right)^{k-1}\begin{bmatrix}
    a & 0 \\ 0 & a
\end{bmatrix}\begin{bmatrix}
    \varphi(b) & 0 \\ \beta & 0
\end{bmatrix}\right) \\
&=& (Tr_2\otimes \varphi)\left(\begin{bmatrix}
    a^k & 0 \\ 0 & a^k 
\end{bmatrix}\begin{bmatrix}
\alpha & \varphi(b) \\ 0 & 0
\end{bmatrix}\left(\begin{bmatrix}\alpha\varphi(b) & \varphi(b^2)\\ \alpha\beta & \beta\varphi(b)\end{bmatrix}\right)^{k-1}\begin{bmatrix}
    \varphi(b) & 0 \\ \beta & 0
\end{bmatrix}\right).
\end{eqnarray*}
We observe that for all $x,y\in\mathbb{C}$, 
$\begin{bmatrix}
    \alpha x & y \\ \alpha\beta & \beta x
\end{bmatrix}=SDS^{-1}$
where 
$$
D=\begin{bmatrix}
    \frac{(\alpha+\beta)x-\gamma}{2} & 0 \\ 0 & \frac{(\alpha+\beta)x+\gamma}{2}
\end{bmatrix}, S=\begin{bmatrix}
    \frac{(\alpha-\beta)x-\gamma}{2\alpha\beta} & \frac{(\alpha-\beta)x+\gamma}{2\alpha\beta} \\ 1 & 1 
\end{bmatrix},\ S^{-1}=\begin{bmatrix}
    \frac{-\alpha\beta}{\gamma} & \frac{\gamma+(\alpha-\beta)x}{2\gamma} \\ 
    \frac{\alpha\beta}{\gamma} & \frac{\gamma-(\alpha-\beta)x}{2\gamma}
\end{bmatrix}
$$
and $\gamma=\sqrt{\varphi(b)^2(\alpha-\beta)^2+4\alpha\beta\varphi(b^2)}$. Then 
\begin{eqnarray*}
    &&\begin{bmatrix}
        \alpha & x \\ 0 & 0 
    \end{bmatrix}\begin{bmatrix}
        \alpha x & y \\ \alpha\beta & \beta x
    \end{bmatrix}^{k-1}\begin{bmatrix}
        x & 0 \\ \beta & 0
    \end{bmatrix} \\
    &=& \begin{bmatrix}
    \frac{(\alpha-\beta)x-\gamma}{2\beta}+x &  \frac{(\alpha-\beta)x+\gamma}{2\beta}+x \\ 0 & 0
\end{bmatrix}\begin{bmatrix}
    \frac{\left((\alpha+\beta)x-\gamma\right)^{k-1}}{2^{k-1}} & 0  \\ 0 &  \frac{\left((\alpha+\beta)x+\gamma\right)^{k-1}}{2^{k-1}}
\end{bmatrix}\begin{bmatrix}
    \frac{-(\beta^2+\alpha\beta)x+\beta\gamma} {2\gamma} & 0 \\ \frac{(\beta^2+\alpha\beta)x+\beta\gamma} {2\gamma} & 0
\end{bmatrix} \\
&=& \frac{1}{2^{k+1}\gamma} \left[\left((\alpha+\beta)x-\gamma\right)^k \left(-(\alpha+\beta)x+\gamma\right)+\left((\alpha+\beta)x+\gamma\right)^k \left((\alpha+\beta)x+\gamma\right) \right]\begin{bmatrix}
    1 & 0 \\ 0 & 0 \end{bmatrix}.
\end{eqnarray*}
Thus, we complete the proof by concluding that
$$
\varphi(p_{\alpha,\beta}^k)
=\frac{\varphi(a^k)}{2^{k+1}\gamma} \left[\left((\alpha+\beta)x+\gamma\right)^{k+1} -\left((\alpha+\beta)x-\gamma\right)^{k+1}  \right].
$$
\end{proof}

\begin{rem}\label{remark general poly}
Let $(\cC,\phi)$ be a $*$-probability space. Let $p$ be any polynomial generated by elements of $\cA$ and $\cB$ with $\cA \prec \mathring{\cB}$. Note that $p$ can be expressed in the form
$$
p=\sum_{j=1}^nb_jy_jb_j'
$$
where $n\geq 1$, $b_1, b_1', \ldots, b_n, b_n' \in alg(\mathcal{B},1_{\cC}), y_1, \ldots, y_n \in \mathcal{Q}$ with $\mathcal{Q}$ being the set defined by   
\begin{equation}\label{class of Poly}
    \mathcal{Q}=\{a_0b_1a_1\cdots b_ka_k \mid k\geq 0, a_0,\dots,a_k\in \cA \text{ and } b_1,\dots,b_k\in\cB \}.
\end{equation} 
To soothe the notation, we write $\Tilde{y}=a_0\phi(b_1)a_1\cdots \phi(b_k)a_k$ whenever $y=a_0b_1a_1\cdots b_ka_k$ for some $a_1,\dots,a_k\in\cA$ and $b_1,\dots,b_k\in \cB$.
Then by monotone independence, we get for any $m \in \mathbb{N}$,
 \begin{equation}\label{eqn: law-general-poly}
     \phi(p^m)=\sum_{j_1,\dots,j_m\in [n]}\phi(\tilde{y}_{j_1}\cdots \tilde{y}_{j_m})\phi(b_{j_1})\phi(b_{j_1}'b_{j_2})\phi(b_{j_2}'b_{j_3})\cdots \phi(b'_{j_{m-1}}b_{j_m})\phi(b_{j_m}').
 \end{equation}  

\end{rem}

We provide in the following proposition an example to demonstrate how to compute the law of a given polynomial. 
\begin{prop}\label{prop: law of general P}
Suppose $(\cC,\varphi)$ is a $*$-probability space, and $a,b_1,b_2,c_1,$ and $c_2$ are elements in $\cC$ such that $(\{a\},\{b_1,b_2,c_1,c_2\})$ is monotone independent. Assume $\varphi(b_1)=\tilde{b}_1, \varphi(b_2)=\tilde{b}_2, \varphi(c_1)=\tilde{c}_1, \varphi(c_2)=\tilde{c}_2$, and 
\begin{equation*} 
\varphi(b_ib_j)=\begin{cases}
b_{1,1} & \text{ if } (i,j)=(1,1), \\
b_{2,2} & \text{ if } (i,j)=(2,2), \\
b_{1,2} & \text{ if } (i,j)=(1,2) \text{ or }(2,1).
\end{cases}
\end{equation*}
We set $p=b_1ac_1ab_1+b_2ac_2ab_2$. Then the distribution of $p$ is 
$$
\alpha\mu_{a^2}\delta_{\vartheta+\zeta}+(1-\alpha)\mu_{a^2}\delta_{\vartheta-\zeta}
$$ 
where $\mu_{a^2}$ is the distribution of $a^2$, $\delta$ denotes the dirac measure, $\alpha=\frac{1}{2\zeta}(\zeta-\vartheta+\tilde{c}_1\tilde{b}_1^2+\tilde{c}_2\tilde{b}_1^2)$ with
$$
 \vartheta=\frac{1}{2}(b_{2,2}\tilde{c}_2+b_{1,1}\tilde{c}_1) \text{ and }\zeta=\frac{1}{2}\sqrt{(b_{2,2}\tilde{c}_2-b_{1,1}\tilde{c}_1)^2+4b_{1,2}^2\tilde{c}_2\tilde{c}_1}.
$$ 
\end{prop}
\begin{proof}
By Remark \ref{remark general poly}, we have 
\begin{eqnarray*}
    \varphi(p^k)
    &=&\varphi\Big([b_1\cdot a\varphi(c_1)a\cdot b_1+b_2\cdot a\varphi(c_2)a\cdot b_2]^k\Big) \\
             &=&\sum_{j_1,\dots,j_k=1}^2\varphi(a_{j_1}\cdots a_{j_k})\varphi(b_{j_1})\varphi(b_{j_1}b_{j_2})\varphi(b_{j_2}b_{j_3})\cdots \varphi(b_{j_{k-1}}b_{j_k})\varphi(b_{j_k})
\end{eqnarray*}
where $a_1=\varphi(c_1)a^2$ and $a_2=\varphi(c_2)a^2$. Note that for all $(j_1,\dots,j_k)\in~[2]^k$, 
$$
\varphi(a_{j_1}\cdots a_{j_k})=\varphi(a^{2k})\varphi(c_{j_1})\varphi(c_{j_2})\cdots \varphi(c_{j_k}).$$ Thus, if for each $i\in[2]$, we set $\alpha_1^{(i)}=\phi(b_i)$, and  
$$
\alpha_k^{(i)} = \sum_{j_1,\dots,j_{k-1}\in [2]}\varphi(b_i\cdot b_{j_1})\varphi(b_{j_1}b_{j_2})\cdots \varphi(b_{j_{k-2}}b_{j_{k-1}})\varphi(b_{j_{k-1}})\varphi(c_{j_1})\varphi(c_{j_2})\cdots \varphi(c_{j_{k-1}})
\text{ for }k\geq 2
$$ 
then we obtain
\begin{eqnarray*}
 \varphi(p^k)&=& \varphi(a^{2k})\alpha_k \quad\text{ where }\quad \alpha_k:=\varphi(b_1)\varphi(c_1)\alpha_k^{(1)}+\varphi(b_2)\varphi(c_2)\alpha_k^{(2)}.   
\end{eqnarray*}
We observe that for each $l\geq 2$,
\[
\alpha_{l}^{(2)}= \varphi(b_2b_1)\varphi(c_1)\alpha_{l-1}^{(1)}+\varphi(b_2b_2)\varphi(c_2)\alpha_{l-1}^{(2)} = b_{1,2}\tilde{c}_1\alpha_{l-1}^{(1)}+b_{2,2}\tilde{c}_2\alpha_{l-1}^{(2)},
\]
which when iterated we obtain for each $l\geq 3$,
\begin{eqnarray*}
\alpha_l^{(2)}
&=& b_{1,2}\tilde{c}_1\alpha_{l-1}^{(1)}+b_{1,2}\tilde{c}_1b_{2,2}\tilde{c}_2 \alpha_{l-2}^{(1)} + b_{2,2}^2\tilde{c}_2^2\alpha_{l-2}^{(2)} \\ 
&=& b_{1,2}\tilde{c}_1\Big(\alpha_{l-1}^{(1)} + b_{2,2}\tilde{c}_2\alpha_{l-2}^{(1)}+b_{2,2}^2\tilde{c}_2^2\alpha_{l-3}^{(1)}\Big)+ b_{2,2}^3\tilde{c}_2^3\alpha_{l-3}^{(2)} \\
&=& b_{1,2}\tilde{c}_1\Big(\alpha_{l-1}^{(1)} + \sum_{j=1}^{l-3}b_{2,2}^j\tilde{c}_2^j\alpha_{l-(j+1)}^{(1)}\Big)+ b_{2,2}^{l-2}\tilde{c}_2^{l-2}(b_{1,2}\tilde{b}_1\tilde{c}_1+b_{2,2}\tilde{b}_2\tilde{c}_2).
\end{eqnarray*}
Using the expression obtained for $\alpha_{k-1}^{(2)}$, we have get for any $k\geq 2$,
\begin{align*}
\alpha_k^{(1)} &= \varphi(b_1b_1)\varphi(c_1)\alpha_{k-1}^{(1)} + \varphi(b_1b_2)\varphi(c_2)\alpha_{k-1}^{(2)} \\
&= b_{1,1}\tilde{c}_1\alpha_{k-1}^{(1)} + b_{1,2}\tilde{c}_2\alpha_{k-1}^{(2)} \\
&= b_{1,1}\tilde{c}_1 \alpha_{k-1}^{(1)}+b_{1,2}^2 \tilde{c}_1 \Big(\tilde{c}_2 \alpha_{k-2}^{(1)}+ \sum_{j=1}^{k-4}b_{2,2}^{j}\tilde{c}_2^{j+1}\alpha_{k-j-2}^{(1)} 
+b_{2,2}^{k-3}\tilde{c}_2^{k-2}\tilde{b}_1 \Big) \\
 & \quad +b_{1,2}b_{2,2}^{k-2}\tilde{c}_2^{k-1}\tilde{b}_2.
 \end{align*}
Then we observe
\begin{eqnarray*}
    \alpha_k^{(1)}-b_{2,2}\tilde{c}_2\alpha_{k-1}^{(1)} = b_{1,1}\tilde{c}_1\alpha_{k-1}^{(1)}+(b_{1,2}^2-b_{2,2}b_{1,1})\tilde{c}_2\tilde{c}_1 \alpha_{k-2}^{(1)},
\end{eqnarray*}
and hence we obtain the following linear homogeneous recurrences
$$
\alpha_k^{(1)}-(b_{2,2}\tilde{c}_2+b_{1,1}\tilde{c}_1)\alpha_{k-1}^{(1)}+(b_{2,2}b_{1,1}-b_{1,2}^2)\tilde{c}_2\tilde{c}_1\alpha_{k-2}^{(1)}=0. 
$$
Thus by solving the characteristic equation $x^2-(b_{2,2}\tilde{c}_2+b_{1,1}\tilde{c}_1)x+(b_{2,2}b_{1,1}-b_{1,2}^2)\tilde{c}_2\tilde{c}_1=~0,$ we have $x=\vartheta\pm\zeta$ as a solution with $$\vartheta=\frac{1}{2}(b_{2,2}\tilde{c}_2+b_{1,1}\tilde{c}_1) \text{ and }\zeta=\frac{1}{2}\sqrt{(b_{2,2}\tilde{c}_2-b_{1,1}\tilde{c}_1)^2+4b_{1,2}^2\tilde{c}_2\tilde{c}_1}.
$$ This yields that 
$$
\alpha_k^{(1)}=s_1^{(1)} (\vartheta+\zeta)^k + s_2^{(1)} (\vartheta-\zeta)^k
$$
for some constants $s_1^{(1)}$ and $s_2^{(1)}$ that determined by the initial conditions, and similarly
$$
\alpha_k^{(2)}=s_1^{(2)}(\vartheta+\zeta)^k+s_2^{(2)}(\vartheta-\zeta)^k
$$
for some constants $s_1^{(2)}$ and $s_2^{(2)}$. Hence, putting the above terms together, we infer that for any $k\geq 1$,
$$
\alpha_k=(\tilde{b}_1\tilde{c}_1s_1^{(1)}+\tilde{b}_2\tilde{c}_2s_1^{(2)})(\vartheta+\zeta)^k+ (\tilde{b}_1\tilde{c}_1s_2^{(1)}+\tilde{b}_2\tilde{c}_2s_2^{(2)})(\vartheta-\zeta)^k.  
$$
Now we observe that $\alpha_0=1$ and $\alpha_1=\tilde{c}_1\tilde{b}_1^2+\tilde{c}_2\tilde{b}_2^2$. Thus,  
\begin{equation*}
\begin{cases}
\tilde{b}_1\tilde{c}_1(s_1^{(1)}+s_2^{(1)})+\tilde{b}_2\tilde{c}_2(s_1^{(2)}+s_2^{(2)}) = 1 \\
\tilde{b}_1\tilde{c}_1[s_1^{(1)}(\vartheta+\zeta)+s_2^{(1)}(\vartheta-\zeta)]+\tilde{b}_2\tilde{c}_2[s_1^{(2)}(\vartheta+\zeta)+s_2^{(2)}(\vartheta-\zeta)]= \tilde{c}_1\tilde{b}_1^2+\tilde{c}_2\tilde{b}_2^2.
\end{cases}
\end{equation*}
As a result, we obtain
\begin{eqnarray*}
\tilde{b}_1\tilde{c}_1s_1^{(1)}+\tilde{b}_2\tilde{c}_2s_1^{(2)} &=& \frac{1}{2\zeta}(\zeta-\vartheta+\tilde{c}_1\tilde{b}_1^2+\tilde{c}_2\tilde{b}_2^2) \\
\tilde{b}_1\tilde{c}_1s_2^{(1)}+\tilde{b}_2\tilde{c}_2s_2^{(2)} &=& \frac{1}{2\zeta}(\zeta+\vartheta-\tilde{c}_1\tilde{b}_1^2-\tilde{c}_2\tilde{b}_2^2)
\end{eqnarray*}
by solving a system of equations which concludes the proof by obtaining for any $k \geq 1$,
\begin{eqnarray*}
\varphi(p^k)&=&\varphi(a^{2k})\alpha_k \\
&=&\frac{\varphi(a^{2k})}{2\zeta}(\zeta-\vartheta+\tilde{c}_1\tilde{b}_1^2+\tilde{c}_2\tilde{b}_2^2) (\vartheta+\zeta)^k + \frac{\varphi(a^{2k})}{2\zeta}(\zeta+\vartheta-\tilde{c}_1\tilde{b}_1^2-\tilde{c}_2\tilde{b}_2^2)(\vartheta-\zeta)^k.
\end{eqnarray*}
\end{proof}

\section{Distribution of Polynomials in Infinitesimally Monotone Independent Elements}
Let $(\cC,\varphi,\varphi')$ be an infinitesimal $*$-probability space and consider its associated  $\widetilde{\mathbb{C}}$-valued probability space $(\widetilde{\cC},\widetilde{\bC},\widetilde{\varphi})$, described in Subsection \ref{subsection: inf Monotone}. Note that by Theorem \ref{thm: inf-M vs OV-M} and the fact that $\widetilde{\bC}$ is commutative, we extend our approach to the framework of $(\widetilde{\cC},\widetilde{\bC},\widetilde{\phi})$. This allows passing directly many results on $*$-distributions of polynomials in monotone independent variables to the infinitesimal setting. 

For two infinitesimally monotone elements $a$ and $b$, we start by illustrating, in Section \ref{Section inf linear span}, how this allows us to derive an explicit expression of the infinitesimal distribution of linear spans of $ab$ and $ba$, covering the cases of the commutator and anti-commutator. Then, we discuss in Remark \ref{remark inf general poly}, more general polynomials in infinitesimal monotone independent elements.  

\medskip
\paragraph{\bf Linear spans in $ab$ and $ba$.} \label{Section inf linear span}
The commutativity of $\widetilde{\mathbb{C}}$ allows deriving explicit expressions of the infinitesimal $*$-distribution of linear spans. Let $a$ and $b$ be two elements in $\cC$ such that $a\pprec b$ and  set $A=\diag(a,a)$ and $B=~\diag(b,b)$. Then, by Theorem \ref{thm: inf-M vs OV-M}, we have that $A\prec B$ with respect to $\widetilde{\varphi}$. We obtain, following the same lines of proof as in Theorem \ref{thm: M-linear span}, the explicit distribution of the linear span $P_{\alpha,\beta}:=P_{\alpha,\beta}(A,B)=\alpha AB+\beta BA$ for any non-zero complex numbers $\alpha$ and $\beta$. More precisely, setting $X=\widetilde{\varphi}(B)$ and $\Gamma=\sqrt{(\alpha-\beta)^2\widetilde{\varphi}(B)^2+4\alpha\beta\widetilde{\varphi}(B^2)}$, we obtain for each $k\geq 1$,
    \begin{eqnarray}\label{formula:OV-span}
\widetilde{\varphi}(P_{\alpha,\beta}^k)  
=\frac{\widetilde{\varphi}(A^k)\Gamma^{-1}}{2^{k+1}} \left[\left((\alpha+\beta)X+\Gamma\right)^{k+1}-\left((\alpha+\beta)X-\Gamma\right)^{k+1} \right].
    \end{eqnarray}
The key point that enables deriving such an expression following the same lines as in the scalar case is the fact that $\widetilde{\bC}$ is commutative, and more precisely,  $X\Gamma= \Gamma X$. Note that by comparing the $(1,2)$-entry of both sides of \eqref{formula:OV-span}, we obtain explicitly the $k$th moment of $p_{\alpha,\beta}$ with respect to $\varphi'$. 

\begin{rem}
Note that we can also compute $\varphi'(p_{\alpha,\beta}^k)$ by formally differentiating $\varphi(p_{\alpha,\beta}^k)$ in Theorem \ref{thm: M-linear span}. However, we adopt the method in \cite{TS19,PT21} by translating problems in the infinitesimal setting to independence over the commutative algebra $\widetilde{\mathbb{C}}$. This methodology proves beneficial for addressing a range of other problems in infinitesimal independence, which motivates our choice to illustrate it for computing the infinitesimal distribution.
\end{rem}

\begin{thm}\label{thm: inf-linear span-M}
Let $a$ and $b$ be two elements in an infinitesimal $*$-probability space $(\cC,\phi,\phi')$ such that $\cA_a\pprec \mathring{\cA}_b$. For given non-zero complex numbers $\alpha$ and $\beta$, let $p_{\alpha,\beta}=\alpha ab+\beta ba$. Then for each $k\geq 0$,
\begin{multline*}
    \varphi'(p_{\alpha,\beta}^k) =
\frac{\phi(a^k)}{2^{k+1}\gamma}\Bigg[((\alpha+\beta)\varphi(b)+\gamma)^k\Big((\alpha+\beta)\big((k+1)\phi'(b)-\varphi(b)\omega/\gamma^2\big)+k\omega/\gamma \Big) \\
- ((\alpha+\beta)\varphi(b)-\gamma)^k\Big((\alpha+\beta)\big((k+1)\phi'(b)-\varphi(b)\omega/\gamma^2\big)-k\omega/\gamma  \Big)
\Bigg] \\
+ \frac{\varphi'(a^k)}{2^{k+1}\gamma}\Big[((\alpha+\beta)\varphi(b)+\gamma)^{k+1}-((\alpha+\beta)\varphi(b)-\gamma)^{k+1}\Big],
\end{multline*}
where 
\begin{equation}\label{gamma and omega}
\ \gamma=\sqrt{(\alpha-\beta)^2\varphi(b)^2+4\alpha\beta\varphi(b^2)}\quad \text{and} \quad
\omega=(\alpha-\beta)^2\varphi(b)\varphi'(b)+2\alpha\beta \varphi'(b^2).
\end{equation}
\end{thm}

\begin{rem}
Note that the infinitesimal $*$-distribution of $p_{\alpha,\beta}$ only depends on that of $a$ and  on $\phi(b),\phi(b^2),\phi'(b)$ and $\phi'(b^2)$. In the particular case where $\phi(b)=\phi'(b)=0$, this reduces to  
$$
\phi'(p_{\alpha,\beta}^k)=\Big[\frac{k}{4}\phi'(b^2)\phi(b^2)^{k/2-1}\phi(a^k)+\frac{\varphi'(a^k)\varphi(b^2)^{k/2}}{2}\Big](\alpha\beta)^{k/2}\Big(1+(-1)^k\Big).
$$
\end{rem}

Again, as a direct consequence of Theorem \ref{thm: inf-linear span-M}, we obtain the infinitesimal $*$-distribution of anti-commutators and commutators by taking $(\alpha,\beta)=(1,1)$ or $(\alpha,\beta)=(i,-i)$ respectively.  
\begin{cor}
Let $p=ab+ba$ and $q=i(ab-ba)$, then for each $k\geq 1$, we have 
\begin{eqnarray*}
\phi'(p^k)
&=&\frac{\phi(a^k)}{4\sqrt{\phi(b^2)}\phi(b^2)}\Bigg[(\phi(b)+\sqrt{\phi(b^2)})^k\Big(2(k+1)\phi'(b)\phi(b^2)-\phi'(b^2)\big[\phi(b)-k\sqrt{\phi(b^2)}\big]\Big) \\
&&+(\phi(b)-\sqrt{\phi(b^2)})^k\Big(2(k+1)\phi'(b)\phi(b^2)-\phi'(b^2)\big[\phi(b)+k\sqrt{\phi(b^2)}\big]\Big)\Bigg] \\
&&+ \frac{\varphi'(a^k)}{2\sqrt{\varphi(b^2)}}\left([\varphi(b)+\sqrt{\varphi(b^2)}]^{k+1}-[\varphi(b)-\sqrt{\varphi(b^2)}]^{k+1}\right)
\end{eqnarray*}
and
\begin{eqnarray*}
    \phi'(q^k)&=& \frac{\phi(a^k)k}{4}\left(1+(-1)^k\right)\left(\sqrt{\phi(b^2)-\phi(b)^2}\right)^{k-2}\left(\phi'(b^2)-2\phi(b)\phi'(b)\right) \\
    &&+\frac{\varphi'(a^k)}{2}\left(1+(-1)^{k}\right)\left(\sqrt{\varphi(b^2)-\varphi(b)^2}\right)^{k}.
\end{eqnarray*}
\end{cor}
\begin{proof}[Proof of Theorem \ref{thm: inf-linear span-M}]
We start by noting that 
$$
\widetilde{\varphi}(P_{\alpha,\beta}^k)=\widetilde{\varphi}\left(\begin{bmatrix} p_{\alpha,\beta}^k & 0 \\ 0 & p_{\alpha,\beta}^k
\end{bmatrix}\right)=\begin{bmatrix}
    \varphi(p_{\alpha,\beta}^k) & \varphi'(p_{\alpha,\beta}^k) \\ 0 & \varphi(p_{\alpha,\beta}^k)
\end{bmatrix}. 
$$
Thus in order to derive an expression for $ \varphi'(p_{\alpha,\beta}^k)$, we shall compute the $(1,2)$-entry of the right hand side of \eqref{formula:OV-span}. In order to do this, let us note  for any $z,z' \in \mathbb{C}$ and matrix $M=\begin{bmatrix}
    z & z' \\ 0 & z
\end{bmatrix}\in\widetilde{\bC}$,
$$
M^{1/2}=\sqrt{M}= \begin{bmatrix}
    \sqrt{z} & \frac{z'}{2\sqrt{z}} \\ 0 & \sqrt{z}
\end{bmatrix} \text{ and }M^{-1}= \begin{bmatrix}
    z^{-1} & -z^{-2}z' \\ 0 & z^{-1}
\end{bmatrix}. 
$$ Setting $x=\varphi(b)$ and $y=\varphi'(b)$, we observe that 
\[
X=\widetilde{\varphi}(B)= \begin{bmatrix}
    x & y \\ 0 & x
\end{bmatrix}
\quad \text{and} \quad  
    \Gamma =\sqrt{(\alpha-\beta)^2\widetilde{\varphi}(B)^2+4\alpha\beta\widetilde{\varphi}(B^2)} \\
        = \begin{bmatrix}
        \gamma & \frac{\omega}{\gamma} \\ 0 & \gamma
    \end{bmatrix},
\]
where $\gamma$ and $\omega$ are defined in \eqref{gamma and omega}, and finally obtain that
$$
\frac{1}{2^{k+1}}\Gamma^{-1}=\frac{1}{2^{k+1}}\begin{bmatrix}
    \gamma^{-1} & -\gamma^{-2}\cdot\omega/\gamma \\ 0 & \gamma^{-1}
\end{bmatrix}= \frac{1}{2^{k+1}\gamma}\begin{bmatrix}
        1 & -\omega/\gamma^2 \\ 0 & 1
    \end{bmatrix}.
$$
Note that 
\begin{eqnarray*}
((\alpha+\beta)X-\Gamma)^{k+1} &=& \begin{bmatrix}
   ((\alpha+\beta)x-\gamma)^{k+1} & (k+1)\big((\alpha+\beta)x-\gamma\big)^{k}\big((\alpha+\beta)y-\omega/\gamma\big) \\ 0 & ((\alpha+\beta)x-\gamma)^{k+1}
\end{bmatrix} \\
&=&  ((\alpha+\beta)x-\gamma)^{k} \begin{bmatrix}
    (\alpha+\beta)x-\gamma & (k+1)\big((\alpha+\beta)y-\omega/\gamma\big) \\ 0 & (\alpha+\beta)x-\gamma
\end{bmatrix}.
\end{eqnarray*}
Therefore, 
\begin{eqnarray*}
&&\frac{1}{2^{k+1}}\Gamma^{-1}\left((\alpha+\beta)X-\Gamma\right)^{k+1} \\
&=& \frac{((\alpha+\beta)x-\gamma)^k}{2^{k+1}\gamma} \begin{bmatrix}
    1 & -\omega/\gamma^2 \\ 0 & 1
\end{bmatrix} \begin{bmatrix}
    (\alpha+\beta)x-\gamma & (k+1)((\alpha+\beta)y-\omega/\gamma) \\ 0 & (\alpha+\beta)x-\gamma
\end{bmatrix} \\
&=& \frac{((\alpha+\beta)x-\gamma)^k}{2^{k+1}\gamma}\begin{bmatrix}
    (\alpha+\beta)x-\gamma & (k+1)((\alpha+\beta)y-\omega/\gamma)-\omega/\gamma^2((\alpha+\beta)x-\gamma) \\ 0 & (\alpha+\beta)x-\gamma
\end{bmatrix} \\
&=& \frac{((\alpha+\beta)x-\gamma)^k}{2^{k+1}\gamma}\begin{bmatrix}
    (\alpha+\beta)x-\gamma & (\alpha+\beta)\Big((k+1)y-x\omega/\gamma^2\Big)-k\omega/\gamma \\ 0 & (\alpha+\beta)x-\gamma
\end{bmatrix}
\end{eqnarray*}
Similarly, we also have
\begin{multline*}
\frac{1}{2^{k+1}}\Gamma^{-1}((\alpha+\beta)X+\Gamma)^{k+1} \\
= \frac{((\alpha+\beta)x+\gamma)^{k}}{2^{k+1}\gamma} \begin{bmatrix}
    (\alpha+\beta)x+\gamma &  (\alpha+\beta)\Big((k+1)y-x\omega/\gamma^2\Big)+k\omega/\gamma  \\ 0 & (\alpha+\beta)x+\gamma
\end{bmatrix}.
\end{multline*}
Thus, the right hand side of Equation \eqref{formula:OV-span} is 
 \begin{multline*}
\frac{\phi(a^k)}{2^{k+1}\gamma}\Bigg[((\alpha+\beta)\varphi(b)+\gamma)^k\Big((\alpha+\beta)\big((k+1)\phi'(b)-\varphi(b)\omega/\gamma^2\big)+k\omega/\gamma \Big) \\
- ((\alpha+\beta)\varphi(b)-\gamma)^k\Big((\alpha+\beta)\big((k+1)\phi'(b)-\varphi(b)\omega/\gamma^2\big)-k\omega/\gamma  \Big)
\Bigg] \\
+ \frac{\varphi'(a^k)}{2^{k+1}\gamma}\Big[((\alpha+\beta)\varphi(b)+\gamma)^{k+1}-((\alpha+\beta)\varphi(b)-\gamma)^{k+1}\Big]
\end{multline*}
which completes the proof.
\end{proof}
\begin{rem}\label{remark inf general poly} \textit{(General Polynomials)}
The commutativity of $\widetilde{\mathbb{C}}$ allows again deriving infinitesimal $*$-distributions of general polynomials. Following the notation in Remark \ref{remark general poly} and assuming $\cA \pprec \mathring{\cB}$, let $p$ be any polynomials  of the form $p=\sum_{j=1}^nb_jy_jb_j'$ with $b_1,b_1',\dots,b_n,b_n'\in\cB$ and $y_1,\dots,y_n\in \mathcal{Q}$. We denote by \[
\widetilde{\mathcal{Q}}=\{A_0B_1A_1\cdots B_kA_k \mid k\geq 0, a_0,\dots,a_k\in \cA \text{ and } b_1,\dots,b_k\in\cB \}
\]
where  
$
A_i=\diag (a_i ,a_i)$  and $ B_j=\diag(b_j, b_j)$ for each $0\leq i\leq n, 1\leq j\leq n$. For each element $Y=A_0B_1A_1\cdots B_kA_k$ in $\widetilde{\mathcal{Q}}$, we set $\widetilde{Y}=A_0\widetilde{\phi}({B}_1)A_1\cdots \widetilde{\phi}({B}_k)A_k$. 
 The for $P=\sum_{j=1}^nB_jY_jB_j'$, we have by Remark \ref{remark general poly} for any $m \in \mathbb{N}$,
\begin{eqnarray}\label{formula:tilde-phi-p^m}
\widetilde{\phi}(P^m) &=&\widetilde{\phi}(\big[\sum_{j=1}^n B_j\widetilde{Y_j}B_j' \big]^m) \nonumber \\ 
&=& \sum_{j_1,\dots,j_m\in [n]}\widetilde{\phi}(\widetilde{Y}_{j_1}\cdots \widetilde{Y}_{j_m})\widetilde{\phi}(B_{j_1})\widetilde{\phi}(B_{j_1}'B_{j_2})\widetilde{\phi}(B_{j_2}'B_{j_3})\cdots \widetilde{\phi}(B'_{j_{m-1}}B_{j_m})\widetilde{\phi}(B_{j_m}), 
\end{eqnarray}
It remains to track the $(1,2)$-entry of the above equality. In the particular case where $p=\sum_{k=1}^n b_ka_kb_k'$, then 
$$\varphi'(p^m)= \!\! \! \!\sum_{\substack{j_1,\dots,j_m\in [n] \\ 
(u_1,\dots,u_{m+1})\in \mathcal{W}_m}} \! \! \! 
\varphi^{(u_1)}(a_{j_1}a_{j_2}\cdots a_{j_m})\varphi^{(u_2)}(b_{j_1})\varphi^{(u_3)}(b_{j_1}'b_{j_2})\cdots \varphi^{(u_{m})}(b'_{j_{m-1}}b_{j_m})\varphi^{(u_{m+1})}(b'_{j_m}),
$$
where $\phi^{(0)}=\phi$ and $\phi^{(1)}=\phi'$ and for each $m\geq 0$, 
$$
\mathcal{W}_m=\{(1,0,0,\dots,0),(0,1,0,\cdots,0),
\cdots,(0,0,\cdots,0,1) \} \subset \mathbb{R}^{m+1}.
$$
\end{rem}

\section{Applications to Random Matrices}\label{Section: RM example}
We show in the section the utility of our results to derive the distributions of some random matrices with respect to the partial trace. Let us first recall that the partial trace is defined for a given $N\times N$ random matrix $A_N$ by
\begin{equation*}
\psi_N(A_N):= \frac{1}{N_0}(\E\circ Tr_{N_0})(A^{(1,1)})
\end{equation*}
where $Tr_{N_0}$ is the non-normalized trace and $A^{(1,1)}$ is the first block of $A_N$. To draw the connection to random matrices, we recall the notation in Subsection \ref{subsection 2.3} and the results stated therein and prove the following:
 
\begin{prop}\label{prop: Lemma for Section 5}
Suppose $A_N$ and $B_N$ are independent $N\times N$ Wigner matrices as defined in Section \ref{subsection 2.3}, then $(T_{Q(A_N)},B_N)_N$ is asymptotically monotone independent with respect to $\psi_N$ for any polynomial $Q$ whose constant term is zero.  
\end{prop}
\begin{proof}
Let $Q$ be a polynomial that has no constant term. We note that $A_N$ and $B_N$ are asymptotically free with respect to $\E\circ Tr_N$ where we recall that $tr_N$ is the normalized trace. Also note that $T_{Q(A_N)}=j_NQ(A_N)(1-j_N)+(1-j_N)Q(A_N)j_N$ where $j_N=\sum_{j=1}^{N_0}E_{N}^{j,j}$ and that by Theorem \ref{thm:inf free - WEJ}, $Q(A_N)$ and $B_N$ are asymptotically infinitesimally free from $j_N$ . Then it follows by Theorem \ref{thm: Inf idem to M} that $(T_{Q(A_N)},B_N)_N$ is asymptotically monotone independent with respect to $\psi_N$.
\end{proof}

\begin{prop}\label{prop Wigner LS}
Let $A_N$ and $B_N$ be independent $N\times N$ Wigner matrices and let for some $m\geq 0$ and non-zero $\alpha\in\mathbb{C}$, $P_N$ be the polynomial defined by 
\begin{equation}\label{poly: self-P}
P_N=\alpha T_{A_N^m}B_N+\widebar{\alpha} B_NT_{A_N^m}.
\end{equation}Then the limiting distribution of $P_N$ with respect to $\psi_N$ is $\mu=\frac{1}{2}(\delta_{\omega}+\delta_{-\omega})$ where $\delta$ is the dirac measure and $\omega=\sqrt{|\alpha|D_m}$ with $$
D_m = \begin{cases}C_{m}-C^2_{m/2} & \text{ if }  m\text{ is even,} \\
            C_m  & \text{ if }  m \text{ is odd,} 
        \end{cases}
$$
and $C_k$ being the $k$-th Catalan number.
\end{prop}

\begin{proof}
By Proposition \ref{prop: Lemma for Section 5}, $(T_{A_N^m},B_N)_N$ is asymptotically monotone independent with respect to $\psi_N$. In addition, we note that the limiting distribution of $T_{A_N^m}$ with respect to $\psi_N$ can be easily computed following \cite[Proposition 5.3]{MT23}. Indeed, we have for any $k\geq 1$,
\begin{eqnarray*}
\lim_{N\to \infty}\psi_N(T_{A_N^m}^k) 
&=& 
\begin{cases}\Big[\lim\limits_{N\to \infty}\big( (\E\circ Tr_N)(A_N^{2m})-(\E\circ Tr_N)(A_N^m)^2\big)\Big]^{k/2} & \text{ if }  k \text{ is even,} \\
            0  & \text{ if }  k \text{ is odd.} 
        \end{cases}
\end{eqnarray*}
By the Wigner semicircular law, we infer that 
\begin{eqnarray*}
\lim_{N\to \infty}\psi_N(T_{A_N^m}^k) 
&=& 
\begin{cases} D_m^{k/2} & \text{ if }  k \text{ is even,} \\
            0  & \text{ if }  k \text{ is odd.} 
        \end{cases}
\end{eqnarray*}
Finally, observing that $\lim_{N\to\infty}(\E\circ Tr_N)(B_N)=0$ and  $\lim_{N\to\infty}(\E\circ Tr_N)(B_N^2)=1$, we get by applying Theorem \ref{thm: M-linear span} and taking into account Remark \ref{rem: dist of span} that for any $k \geq 1$, 
\begin{eqnarray*}
    \lim_{N\to\infty} \psi_N(P_N^k)=\begin{cases} \big(|\alpha|D_m\big)^{k/2} & \text{ if }k \text{ is even} \\
    0 & \text{ if }k \text{ is odd.}
    \end{cases} 
\end{eqnarray*}

\end{proof} 

\begin{rem}
The above result also holds for $A_N$ being an $N\times N$ GUE matrix, and $B_N$ an entry permutation of $A_N$ that is self-adjoint. In particular, $B_N$ can be chosen to be the transpose of $A_N$. Indeed, by \cite{MP16} and \cite{P22}, we have  that $(A_N,B_N)_N$ is asymptotically free and by combining Theorems \ref{thm: Inf idem to M} and \ref{thm:inf free - WEJ}, we obtain that $(T_{A_N},B_N)_N$ is asymptotically monotone independent with respect to $\psi_N$. Then by analogous arguments, we obtain the same limiting $*$-distribution of $P_N$ with respect to $\psi_N$ as in Proposition \ref{prop Wigner LS}. 
\end{rem}

Finally, we give another example around the polynomial in Proposition \ref{prop: law of general P} and show how our results easily extend to computing the distribution with respect to the partial trace of such a polynomial in Wigner matrices.  
\begin{prop}\label{prop Wigner general poly}
    Let $A_N$ and $B_N$ be $N \times N$ independent Wigner matrices and consider the polynomial given by
$$P_N=B_N^{2n}T_{A_N}B_N^{2h}T_{A_N}B_N^{2n}+B_N^{2m}T_{A_N}B_N^{2s}T_{A_N}B_N^{2m}$$ for some $n,m,h,s\geq 1$.  
Then $P_N$ converges in distribution with respect to $\psi_N$ to $\mu=\alpha\delta_{\vartheta+\zeta}+(1-\alpha)\delta_{\vartheta-\zeta}$ where $\delta$ denotes the dirac measure, $ \alpha=\frac{1}{2\zeta}(\zeta-\vartheta+C_hC_n^2+C_sC_m^2)$
with 
$$
\vartheta=\frac{1}{2}(C_{2m}C_s+C_{2n}C_h), \qquad \zeta=\frac{1}{2}\sqrt{(C_{2m}C_s-C_{2n}C_h)^2+4b_{1,2}^2C_sC_h},
$$
and $C_k$ being the $k$-th Catalan number.
\end{prop}

\begin{proof}
It is notable that $(T_{A_N},B_N)_N$ is asymptotically monotone independent with respect to $\psi_N$ by Proposition \ref{prop: Lemma for Section 5}. Then there is a $*$-probability space $(\cC,\psi)$ and elements $a$ and $b$ in $\cC$ such that  $a\prec b$ and $(T_{A_N},B_N)$ converges to $(a,b)$ in $*$-distribution. In other words, for any $k\geq 1$,
$$
\lim_{N\to\infty} \psi_N(P_N^k)=\psi(p^k)
\text{ where } p=b_1ac_1ab_1+b_2ac_2ab_2. 
$$ 
with $b_1=b^{2n}, b_2=b^{2m}, c_1=b^{2h}, c_2=b^{2s}.$ 
Following the notation in Proposition \ref{prop: law of general P}, we have 
$\tilde{b}_1=C_n, \tilde{b}_2=C_m, \tilde{c}_1=C_h, \tilde{c}_2=C_s$, and 
\begin{equation}\label{formula: psi bb }
b_{1,2}=\begin{cases}
C_{2n} & \text{ if } (i,j)=(1,1), \\
C_{2m} & \text{ if } (i,j)=(2,2), \\
C_{n+m} & \text{ if } (i,j)=(1,2) \text{ or }(2,1).
\end{cases}
\end{equation}

Then by applying Proposition \ref{prop: law of general P}, we conclude that
\begin{equation*}
\psi(p^k)=\frac{1}{2\zeta}(\zeta-\vartheta+C_hC_n^2+C_sC_m^2) (\vartheta+\zeta)^k + \frac{1}{2\zeta}(\zeta+\vartheta-C_hC_n^2-C_sC_m^2)(\vartheta-\zeta)^k
\end{equation*}
where 
$$
\vartheta=\frac{1}{2}(C_{2m}C_s+C_{2n}C_h), \quad \text{ and } \quad \zeta=\frac{1}{2}\sqrt{(C_{2m}C_s-C_{2n}C_h)^2+4C_{n+m}^2C_sC_h}.
$$
\end{proof}

\begin{rem}
$\mu=\alpha\delta_{\vartheta+\zeta}+(1-\alpha)\delta_{\vartheta-\zeta}$ in Proposition \ref{prop Wigner general poly} is a probability measure. One could also check that indeed $\alpha \in [0,1]$ using the fact that
$$
2C_n^2C_m^2\leq C_{n+m}^2\leq C_{2n}C_{2m} \text{ for all }n,m\geq 1. 
$$
Note that the first inequality can be deduced by induction, while the second one follows from the Cauchy-Schwartz inequality with respect to the inner product on polynomials defined by 
$$\langle f,g\rangle:=\int f(x)g(x)d\mu \text{  where  }d\mu(t)=\frac{1}{2\pi}\sqrt{4-t^2}dt.$$ 
\end{rem}

\section*{Acknowledgement}

The authors would like to thank Octavio Arizmendi, Benson Au and Adri\'{a}n Celestino for valuable discussions, and the referee for identifying a mistake in an earlier version as well as for providing insightful comments that enhanced the clarity and exposition of the paper.

\end{document}